\documentclass[11pt]{amsart}

\usepackage{amsfonts,amsmath,amssymb,color,amscd,amsthm}
\usepackage[alphabetic]{amsrefs}

\usepackage[T1]{fontenc}
\usepackage{typearea}
\usepackage{color}

\newtheorem{lemma}{Lemma}[section]

\newtheorem{thm}[lemma]{Theorem}
\newtheorem{proposition}[lemma]{Proposition}

\theoremstyle{definition}

\newtheorem{notation}[lemma]{Notation}

\theoremstyle{remark}
\newtheorem{remark}[lemma]{Remark}

\newcommand\C{{\mathbb C}}
\newcommand\PP{{\mathbb P}}

\title{Counterexamples to the Complement Problem}

\author[P.-M.~Poloni]{Pierre-Marie~Poloni}
\address{Pierre-Marie Poloni \\ Universit\"{a}t Bern \\ Mathematisches Institut \\Sidlerstrasse 5 \\ CH-3012 Bern \\ Switzerland}
\email{pierre.poloni@math.unibe.ch}

\begin{document}

\begin{abstract}We provide explicit counterexamples to the so-called  Complement Problem in every dimension $n\geq3$, i.e.\ pairs of nonisomorphic irreducible algebraic hypersurfaces $H_1, H_2\subset\C^{n}$ whose complements  $\C^{n}\setminus H_1$ and $\C^{n}\setminus H_2$ are isomorphic. Since we can  arrange that one of the hypersurfaces is singular whereas the other is smooth, we also have counterexamples  in the analytic setting.
\end{abstract}

\maketitle

\section{Introduction}   

The  Complement Problem (in the affine $n$-space) is one of the ``challenging problems'' considered by Hanspeter Kraft in his survey on affine algebraic geometry at the  Bourbaki seminar \cite{Kraft}. It is formulated as follows. 
\begin{quote}
Given two irreducible hypersurfaces $H_1,H_2\subset\C^n$ and an isomorphism of their complements, does it follow that $H_1$ and $H_2$ are isomorphic? 
\end{quote}      

 Let us specify  that we work here  in the context  of  algebraic geometry.  In particular, the hypersurfaces considered above are algebraic, i.e.\ defined as the zero sets of some polynomials $f_1,f_2\in\C[x_1,\ldots,x_n]$, and  the isomorphisms are isomorphisms of algebraic varieties. Moreover, we recall that the complement $\C^n\setminus H$ of an hypersurface $H\subset\C^n$   is also an affine algebraic variety.  %Indeed, if $H$ is  defined by the  equation $f(x_1,\ldots,x_n)=0$, then one can realize its complement as the hypersurface of $\C^{n+1}$ that is defined by the equation $x_{n+1}f(x_1,\ldots,x_n)=1$.

The Complement Problem is a very natural question:  We want to retrieve some information   about a subvariety $X\subset M$  from its complement $M\setminus X$.  Such questions make of course sense in various contexts (as e.g.\ in knot theory, see \cite{Cappell-Shaeneson} and \cite{Gordon-Luecke}).  Closer to our immediate interests,  J\'er\'emy Blanc \cite{Blanc}  gave counterexamples to the Complement Problem for curves in the projective plane $\PP^2$. Actually, his main motivation was to disprove another conjecture, due to Hisao  Yoshihara \cite{Yoshihara}, which stated that if two irreducible curves $\Gamma_1,\Gamma_2\subset\PP^2$ have isomorphic complements $\PP^2\setminus\Gamma_1\simeq\PP^2\setminus\Gamma_2$, then they should be equivalent, i.e.\ there should exist an automorphism of $\PP^2$ sending $\Gamma_1$ onto $\Gamma_2$.

The purpose of this note is to answer Kraft's Complement Problem in the negative for every $n\geq3$.  More precisely, we will give explicit counterexamples of several different types, as described in the main theorem  below.

\begin{thm}\label{main-thm}For every integer $n\geq3$, there exist examples of 
\begin{enumerate}
\item irreducible hypersurfaces $H_1,H_2\subset\C^n$ with isomorphic complements $\C^n\setminus H_1 \simeq \C^n\setminus H_2$ such that $H_1$ and $H_2$ are  smooth and nonisomorphic; 
\item irreducible hypersurfaces $H_1,H_2\subset\C^n$ with isomorphic complements $\C^n\setminus H_1 \simeq \C^n\setminus H_2$ such that $H_1$ is smooth but  $H_2$ is singular;
\item irreducible hypersurfaces $H_1,H_2\subset\C^n$ with isomorphic complements $\C^n\setminus H_1 \simeq \C^n\setminus H_2$ such that $H_1$ and $H_2$ are  isomorphic, although there is no automorphism of $\C^n$ mapping $H_1$ onto $H_2$. 
\end{enumerate}
\end{thm}

 At the time that the first version of the present paper was being written,  the case of irreducible curves on $\C^2$ was still wide open. But  since then it has been solved, again in the negative, by  Blanc, Furter and Hemmig in a remarkable paper \cite{BFH} in which they make use of totally different methods to find counterexamples to the Complement Problem in the case where $n=2$.

We remark that the second kind of examples in above Theorem \ref{main-thm} provide  counterexamples to the Complement Problem in the analytic setting too. On the other hand, the nonisomorphic algebraic varieties that we will produce  in case (1) are biholomorphic, and we do not get any examples of  smooth affine algebraic varieties $V_1,V_2\subset\C^n$ (i.e.\  of affine algebraic manifolds in $\C^n$), which are not biholomorphic, although their complements $\C^n\setminus V_1$ and  $\C^n\setminus V_2$ are. By contrast,  it is shown  in \cite{BFH} that if two nonisomorphic irreducible affine curves have isomorphic complements in $\C^2$, then they are necessarily smooth and  biholomorphic.

All our examples will be realized as  hypersurfaces of $\C^{m+2}$ defined by an equation of the form  $x_1^2\cdots x_m^2y+z^2+x_1\cdots x_m(z^2-\alpha)^k=\alpha$ for some integer $k\geq0$ and some constant $\alpha\in\C$. These varieties were first studied  by Lucy Moser-Jauslin and the author in \cite{MJP} for the case where $m=1$ and then in \cite{Poloni-stable} for the general case. In particular, it was  observed that there exist such polynomials, say  $P$ and $Q$, whose zero sets $\{P=0\}$ and $\{Q=0\}$ are not isomorphic, whereas their other fibers $\{P=c\}$ and $\{Q=c\}$ are isomorphic for all $c\in\C^*$. The main ingredient of the present paper  will be to use the isomorphisms $\{P=c\}\simeq\{Q=c\}$ to construct an isomorphism between the complements $\C^{m+2}\setminus\{P=0\}$ and  
$\C^{m+2}\setminus\{Q=0\}$.

\medskip

{\bf Acknowledgement.} The author thanks the referee for helpful comments and suggestions and for pointing out an error in a previous version of Proposition \ref{exple2}. 

\section{Preliminaries}
 
Let us start by recalling some notations and results from \cite{Poloni-stable} that we will use in the sequel.
 
Throughout this paper, we fix an integer $m\geq1$ and a coordinate  system $x_1,\ldots,x_m,y,z$ on the complex  affine space $\C^{m+2}$. If $P\in\C[x_1,\ldots,x_m,y,z]$, then  $V(P)$ denotes  the zero set of $P$ in $\C^{m+2}$.

\begin{notation} Given a polynomial $q(t)\in\C[t]$, we denote by $P_q$ the polynomial of $\C[x_1,\ldots,x_m,y,z]$ defined by $$P_q=x_1^2\cdots x_m^2y+z^2+x_1\cdots x_mq(z^2).$$ 
\end{notation}

 It was shown in \cite{Poloni-stable} that the  algebraic varieties $V(P_{0}), V(P_{0}-1), V(P_{1})$ and $V(P_{1}-1)$ are pairwise nonisomorphic. Moreover,  every fiber $V(P_q-c)=P_q^{-1}(c)\subset\C^{m+2}$  is isomorphic to one of these four  and we have the following classification result.

\begin{proposition}[\cite{Poloni-stable}*{Lemma 2.2 and Proposition 2.5}]\label{prop_isom}
Let $q(t)\in\C[t]$ and  $c\in\C$. Then, the variety $V(P_q-c)$ is isomorphic to  $V(P_{q(c)}-c)$. Moreover, the latter is isomorphic to:
\begin{itemize}
\item  $V(P_{0})$ if and only if $c=0$ and $q(c)=0$;
\item  $V(P_{0}-1)$ if and only if $c\neq0$ and $q(c)=0$;
\item  $V(P_{1})$ if and only if $c=0$ and $q(c)\neq0$;
\item  $V(P_{1}-1)$ if and only if $c\neq0$ and $q(c)\neq0$.
\end{itemize}
\end{proposition}

Finally, we recall the classification of the hypersurfaces $V(P_q-c)\subset\C^{m+2}$ up to equivalence, i.e.\ up to automorphisms of the ambient space. 

\begin{proposition}[\cite{Poloni-stable}*{Proposition 3.2}]\label{prop-classification-equivalence}
Let $q_1(t),q_2(t)\in\C[t]$ be two polynomials and $c_1,c_2\in\C$ be two constants. Then, the following are equivalent. 
\begin{enumerate}
\item There exists an algebraic automorphism of $\C^{m+2}$ which maps the hypersurface $V(P_{q_1}-c_1)$ onto $V(P_{q_2}-c_2)$.
\item There exist $\lambda,\mu\in\C^*$ such that $c_2=\mu^{-1}c_1$ and $q_2(t)=\lambda q_1(\mu t)$.  
\end{enumerate}
\end{proposition}

\section{Explicit examples}

All our examples will consist in hypersurfaces $H_{\alpha,k}$ in $\C^{m+2}$ defined by an equation of the  form
\[x_1^2\cdots x_m^2y+z^2+x_1\cdots x_m(z^2-\alpha)^k=\alpha\]
for some integer $k\geq0$ and some constant $\alpha\in\C$. By  Proposition \ref{prop_isom}, the variety $H_{\alpha,k}=V(P_{(t-\alpha)^k}-\alpha)$ is isomorphic to: 
$$\left\{\begin{array}{ll} 
V(P_{0}) & \text{if  } \alpha=0 \text{ and } k\geq1,\\
V(P_{0}-1) & \text{if  } \alpha\neq0 \text{ and } k\geq1,\\
V(P_{1}) & \text{if  } \alpha=0 \text{ and } k=0,\\
V(P_{1}-1) & \text{if  } \alpha\neq0 \text{ and } k=0.
\end{array}\right.$$

In particular, since $V(P_{0})$, $V(P_{0}-1)$, $V(P_{1})$ and $V(P_{1}-1)$ are pairwise nonisomorphic, we observe that  $H_{\alpha,k}\not\simeq H_{\alpha,0}$ if $k\neq0$.

\begin{lemma}\label{lemme-complement}The hypersurfaces   $H_{\alpha,k}$ and $H_{\alpha,k'}$ have isomorphic complements, i.e.\   $\C^{m+2}\setminus H_{\alpha,k} \simeq \C^{m+2}\setminus H_{\alpha,k'}$ for all $\alpha\in\C$ and all $k,k'\geq0$.  
\end{lemma}

\begin{proof}
Following  the notation of the previous section, we have that  $H_{\alpha,k}=V(P_{q_k}-\alpha)$, where  $q_k(t)=(t-\alpha)^k\in\C[t]$. To prove that  $H_{\alpha,k}$ and $H_{\alpha,k'}$ have isomorphic complements, it suffices to prove that $\C^{m+2}\setminus H_{\alpha,k} \simeq \C^{m+2}\setminus H_{\alpha,0}$ for all $\alpha\in\C$ and all $k\geq1$. We do this by giving an explicit isomorphism.

 We set $P=P_{q_k}-\alpha$ and $Q=P_{q_0}-\alpha$, so that the coordinate rings of $\C^{m+2}\setminus H_{\alpha,k}$ and $\C^{m+2}\setminus H_{\alpha,0}$ are isomorphic to the rings 
 $\C[x_1,\ldots,x_m,y,z,\frac{1}{P}]$ and $\C[x_1,\ldots,x_m,y,z,\frac{1}{Q}]$, respectively.

Next, we consider the morphisms 
\[\begin{array}{rrcl}
\Phi:&\C^{m+2}\setminus H_{\alpha,0}&\to&\C^{m+2}\\
&(x_1,\ldots,x_m,y,z)&\mapsto&\left(\dfrac{x_1}{Q^k},x_2,\ldots,x_m,yQ^{2k}+Q^k\dfrac{Q^k-(z^2-\alpha)^k}{x_1\cdots x_m},z\right)
\end{array}\]
and

\[\begin{array}{rrcl}
\Psi:&\C^{m+2}\setminus H_{\alpha,k}&\to&\C^{m+2}\\
&(x_1,\ldots,x_m,y,z)&\mapsto&\left(P^{k}x_1,x_2,\ldots,x_m,\dfrac{1}{P^{2k}}(y-\dfrac{P^k-(z^2-\alpha)^k}{x_1\cdots x_m}),z\right)
\end{array}.\] 

We remark that the above morphisms are well defined, since $\frac{Q^k-(z^2-\alpha)^k}{x_1\cdots x_m}$ and  $\frac{P^k-(z^2-\alpha)^k}{x_1\cdots x_m}$ are both elements of $\C[x_1,\ldots,x_m,y,z]$.

One checks by a straightforward calculation that $P\circ\Phi=Q$ and $Q\circ\Psi=P$. This shows that $\Phi(\C^{m+2}\setminus H_{\alpha,0})\subset\C^{m+2}\setminus H_{\alpha,k}$ and $\Psi(\C^{m+2}\setminus H_{\alpha,k})\subset\C^{m+2}\setminus H_{\alpha,0}$. Finally, one easily checks that $\Phi\circ\Psi=\text{id}_{\C^{m+2}\setminus H_{\alpha,k}}$ and $\Psi\circ\Phi=\text{id}_{\C^{m+2}\setminus H_{\alpha,0}}$.
 \end{proof}

Combining  Proposition \ref{prop_isom} with Lemma \ref{lemme-complement}, we obtain the following counterexamples to the  Complement Problem.

\begin{proposition}\label{exple1}Let $m\geq1$ and let $H_1$ and $H_2$ be the irreducible hypersurfaces of $\C^{m+2}$ that are defined by the equations 
\[x_1^2\cdots x_m^2y+z^2+x_1\cdots x_m(z^2-1)=1\]
and
\[x_1^2\cdots x_m^2y+z^2+x_1\cdots x_m=1,\]
respectively. Then, $H_1$ and $H_2$ are smooth  and not isomorphic, although they have isomorphic complements  $\C^{m+2}\setminus H_1 \simeq \C^{m+2}\setminus H_2$.  
\end{proposition}

\begin{proof}
On the one hand,  Proposition \ref{prop_isom} implies that the hypersurfaces $H_1\simeq V(P_{0}-1)$ and $H_2= V(P_{1}-1)$ are not isomorphic. On the other hand, their complements are isomorphic by Lemma \ref{lemme-complement}.
\end{proof}

\begin{remark}Even if they are not isomorphic as algebraic varieties, the above hypersurfaces $H_1\simeq V(P_{0}-1)$ and $H_2= V(P_{1}-1)$ are  biholomorphic \cite{Poloni-stable}*{Remark 2.6}.
\end{remark}

We now give counterexamples in the analytic category. For this, we remark that the hypersurface $H_{0,0}$ is smooth  in the case where $m=1$. Nevertheless, by Lemma \ref{lemme-complement}, its complement $\C^3\setminus H_{0,0}$ in $\C^3$ is isomorphic to that of the singular hypersurface $H_{0,1}$.  Considering the cylinders over these two hypersurfaces, we obtain  nonbiholomorphic  counterexamples to the Complement Problem in any dimension $n\geq3$. 

\begin{proposition}\label{exple2}   Let $m=1$ and denote by  $S_1$ and $S_2$  the irreducible hypersurfaces of $\C^{m+2}=\C^{3}$ that are defined by the equations 
\[x_1^2y+z^2+x_1z^2=0\]
and
\[x_1^2y+z^2+x_1=0,\]
respectively. Let $m'\geq0$ be any nonnegative integer and consider the hypersurfaces $H'_1=S_1\times\C^{m'}$ and $H'_2=S_2\times\C^{m'}$ in $\C^{m'+3}$.

Then, the complements  $\C^{m'+3}\setminus H'_1$ and $\C^{m'+3}\setminus H'_2$ are isomorphic. However, since  $H'_1$ is singular and $H'_2$ is smooth, $H'_1$ and $H'_2$ are not biholomorphic.
\end{proposition}

\begin{proof}
It is straightforward to check that  $S_1$   is singular and that  $S_2$  is smooth. Hence, $H'_1$ is singular and $H'_2$ is smooth. Since $S_1=H_{0,1}$ and $S_2=H_{0,0}$, their complements  $\C^{3}\setminus S_1$ and $\C^{3}\setminus S_2$ are isomorphic by Lemma \ref{lemme-complement}. This implies that  $H'_1$ and $H'_2$ have isomorphic complements in $\C^{m'+3}$.
\end{proof}

Let us conclude by giving, thanks to Proposition \ref{prop-classification-equivalence}, an example of two smooth nonequivalent hypersurfaces which are isomorphic and have isomorphic complements.

\begin{proposition}\label{exple3}Let $m\geq1$ and let $H''_1$ and $H''_2$ be the hypersurfaces of $\C^{m+2}$ that are defined by the equations 
\[x_1^2\cdots x_m^2y+z^2+x_1\cdots x_m(z^2-1)=1\]
and
\[x_1^2\cdots x_m^2y+z^2+x_1\cdots x_m(z^2-1)^2=1,\]
respectively. Then, $H''_1$ and $H''_2$ are smooth irreducible varieties which are isomorphic and have isomorphic complements in $\C^{m+2}$. Nevertheless, no automorphisms of $\C^{m+2}$ map $H''_1$ onto $H''_2$.
\end{proposition}

\begin{proof}Proposition \ref{prop-classification-equivalence} shows that the hypersurfaces $H''_1=V(P_{(t-1)}-1)=H_{1,1}$ and $H''_2=V(P_{(t-1)^2}-1)=H_{1,2}$ are not equivalent. Nevertheless, their complements   are isomorphic by Lemma \ref{lemme-complement}. 
\end{proof}

%%%%%%%%%%%%%%%%% References %%%%%%%%%%%%%%%%%%%%%

\end{document}